\documentclass[11pt]{amsart}

\pdfoutput=1

\usepackage[T1]{fontenc}
\usepackage{lmodern}

\usepackage{esint,amssymb} 
\usepackage{graphicx}
\usepackage{MnSymbol}
\usepackage{mathtools} 
\usepackage[colorlinks=true, pdfstartview=FitV, linkcolor=blue, citecolor=blue, urlcolor=blue,pagebackref=false]{hyperref}
\usepackage{microtype}

\usepackage{bm}
\usepackage{dsfont}
\usepackage{mathrsfs}
\usepackage{xcolor}

\newtheorem{proposition}{Proposition}
\newtheorem{theorem}[proposition]{Theorem}
\newtheorem{lemma}[proposition]{Lemma}

\theoremstyle{remark}
\newtheorem{remark}[proposition]{Remark}

\theoremstyle{definition}

\numberwithin{equation}{section}
\numberwithin{proposition}{section}

\newcommand{\N}{\mathbb{N}}

\newcommand{\R}{\mathbb{R}}

\newcommand{\E}{\mathbb{E}}
\renewcommand{\P}{\mathbb{P}}

\newcommand{\ep}{\varepsilon}
\newcommand{\eps}{\varepsilon}

\renewcommand{\le}{\leqslant}
\renewcommand{\ge}{\geqslant}

\renewcommand{\bar}{\overline}
\newcommand{\ov}{\overline}

\newcommand{\td}{\widetilde}
\renewcommand{\hat}{\widehat}

\newcommand{\Ll}{\left}
\newcommand{\Rr}{\right}
\renewcommand{\d}{\mathrm{d}}

\newcommand{\1}{\mathbf{1}}

\newcommand{\mcl}{\mathcal}

\newcommand{\al}{\alpha}
\newcommand{\be}{\beta}
\newcommand{\ga}{\gamma}
\newcommand{\de}{\delta}

\begin{document}

\author{Jean-Christophe Mourrat}
\address[Jean-Christophe Mourrat]{Department of Mathematics, ENS Lyon and CNRS, Lyon, France}

\subjclass[2010]{}
\date{\today}

\title[Operator $\ell^\infty \to \ell^\infty$ norm of products of random matrices]{Operator $\ell^\infty \to \ell^\infty$ norm of products of random matrices with iid entries}

\begin{abstract}
We study the $\ell^\infty \to \ell^\infty$ operator norm of products of independent random matrices with independent and identically distributed entries. For $n$-by-$n$ matrices whose entries are centered, have unit variance, and have a finite moment of order $4\al$ for some $\al > 1$, we find that the operator norm of the product of $p$ matrices behaves asymptotically like $n^{\frac {p+1}{2}}\sqrt{2/\pi}$. The case of products of possibly non-square matrices with possibly non-centered entries is also covered. 
\end{abstract}

\maketitle

%
%
%
%
%
%
\section{Introduction and main results}

The goal of this paper is to investigate the $\ell^\infty \to \ell^\infty$ norm of products of random matrices with independent and identically distributed (i.i.d.)\ entries. We will identify, in the limit of large matrix sizes, the asymptotic behavior of this operator norm. 

One motivation for studying this question comes from the study of neural networks. One simple neural-network architecture consists in the recursive iteration of a multiplication by a matrix and a component-wise non-linearity such as the positive part $x \mapsto \max(0,x)$. Before training, the matrix entries are typically initialized as independent Gaussian random variables. If we ignore the non-linearity and consider the neural network at initialization, then the mapping from input to output becomes linear and is described by a product of random matrices with i.i.d.\ entries. For more realistic situations, one may get inspiration from the results presented here to devise heuristic guesses for what the $\ell^\infty \to \ell^\infty$ Lipschitz norm of a neural network might be. 

The Lipschitz norm of a neural network is tightly linked with its robustness properties under adversarial attacks. Indeed, the concept of an adversarial attack is often phrased as a small perturbation of the input that causes a major change in the output \cite{szegedy2013intriguing}, with some work suggesting that this is due to the approximately linear behavior of the network \cite{goodfellow2014explaining}. In this context, measuring the sizes of these changes in the $\ell^\infty$ norm is rather natural, see for instance \cite{goodfellow2014explaining, madry2017towards}. A series of works including \cite{bhowmick2021lipbab, gonon2023path, herrera2020local, shi2022efficiently, virmaux2018lipschitz} aim to compute, approximate or bound the Lipschitz constant of neural-network architectures. We also refer to \cite{bartlett2017spectrally, gonon2023approximation, neyshabur2018pac} on the interplay between Lipschitz estimates and generalization bounds.

By duality, the results presented here also yield the asymptotic behavior of the $\ell^1 \to \ell^1$ operator norm of a product of random matrices with i.i.d.\ entries. The $\ell^\infty \to \ell^\infty$ operator norm of a matrix $A$ has a very simple representation as the maximum of the $\ell^1$ norms of the rows of $A$, and this will be the starting point of our analysis. Similar explicit representations exist for the $\ell^\rho \to \ell^q$ operator norm if $\rho = 1$ or $q = \infty$, and it may be possible to extend the asymptotic analysis to these cases. Another notable case is when $\rho = q = 2$, which can be approached using spectral techniques. In particular, for a product of $p$ $n$-by-$n$ matrices with i.i.d.\ centered entries with unit variance, the $\ell^2 \to \ell^2$ operator norm behaves asymptotically like $(p+1)n^{\frac p 2}$ \cite{alexeev2010asymptotic, bai1986limiting}\footnote{The upper bound follows from a minor modification of the proof of \cite[Theorem~2.1]{bai1986limiting}, while the lower bound is a consequence of \cite[Theorem~1.1 and Remark~1.3]{alexeev2010asymptotic}. With non-square matrices, I do not expect there to be a simple formula for the limit, except in terms of solutions to the polynomial equation in \cite[(1.2)]{alexeev2010asymptotic}.} (see also \cite{dubach2021words} for the determination of the limit spectral measure). For other values of $\rho$ and $q$, the asymptotic behavior of the operator norm of a single matrix with i.i.d.\ bounded entries, up to a multiplicative constant,  is discussed in \cite{bennett1977schur, carl1978grenz} and \cite[Remark~4.2]{adamczak2024norms}. There was recent effort to capture the $\ell^\rho \to \ell^q$ operator norm of a matrix with i.i.d.\ or structured independent entries in terms of simpler quantities, up to absolute multiplicative constants (in particular up to constants that do not depend on $\rho$ or $q$) \cite{adamczak2024norms, guedon2017expectation, latala2023chevet, latala2024operator, latala2025operator}. 

We remark that, by interpolation and using the known asymptotic behavior of the $\ell^2 \to \ell^2$ operator norm of a product of random matrices, our results also yield an explicit estimate on the limsup of the $\ell^\rho \to \ell^\rho$ norm of a product of random matrices, for every $\rho \in [1,+\infty]$ (using also duality to cover the interval $[1,2]$). 

An upper bound on the operator norm of a product of matrices is given by the product of the operator norms of the individual matrices. Such an upper bound is often used in practical settings \cite{bartlett2017spectrally, golowich2018size, gonon2023approximation, neyshabur2018pac, neyshabur2015norm}, and is sharp in general. However, for the $\ell^\infty \to \ell^\infty$ operator norm of a product of~$p$ random matrices, this upper bound fails to capture the correct order of magnitude, as it scales like $n^p$ instead of $n^{\frac {p+1}{2}}$ (see Proposition~\ref{p.one.matrix} and Theorem~\ref{t.centered}). This is in contrast to the $\ell^2 \to \ell^2$ operator norm, for which the two quantities only differ by a constant factor $(2^p$ in place of $p+1$ for a product of $p$ $n$-by-$n$ matrices). 

We now proceed to present the main results of this paper more precisely.
For every $x = (x_1,\ldots, x_n) \in \R^n$, we denote by $|x|_{\infty} := \max(|x_1|, \ldots, |x_n|)$  the $\ell^\infty$ norm of $x$. For every matrix $A = (A_{ij})_{1 \le i \le m, 1 \le j \le n} \in \R^{m\times n}$, we denote the $\ell^\infty \to \ell^\infty$ operator norm of $A$ by
\begin{equation}  
\label{e.def.norm}
\|A\|_{\ell^\infty \to \ell^\infty} := \sup_{|x|_{\infty} \le 1} |Ax|_{\infty} = \max_{1 \le i \le m} \sum_{j = 1}^n |A_{ij}|. 
\end{equation}
For completeness, we start by recording the elementary case of a single matrix.
\begin{proposition}[single matrix]
Let $\al > 1$ and let $A = (A_{ij})_{1 \le i \le m, 1 \le j \le n}$ be a family of i.i.d.\ random variables with finite moment of order $2\al$. As $m$ and $n$ tend to infinity with $m = O(n)$, we have
\begin{equation*}  
n^{-1} \|A\|_{\ell^\infty \to \ell^\infty} \xrightarrow{\text{(prob.)}} \E[|A_{11}|].
\end{equation*}
\label{p.one.matrix}
\end{proposition}
\begin{remark}  
In Proposition~\ref{p.one.matrix} and throughout, we implicitly understand that the law of one entry of the matrix $A$ is kept fixed as $m$ and $n$ are sent to infinity.
\end{remark}
\begin{remark}
The moment assumption on the entries of $A$ appearing in Proposition~\ref{p.one.matrix} is essentially necessary, in the sense that for every $\ep \in (0,1]$, one can find a distribution with finite moment of order $2-2\ep$ such that whenever $n = O(m)$, one of the entries of the matrix $A$ is of the order of $n^\frac 2{2-\ep} \gg n$.
A similar argument can be used to show the necessity of the assumption that $m = O(n)$.
\end{remark}

In order to state our results concerning products of matrices, we introduce some notation. For a family of random variables $(U(n_0,\ldots, n_p))_{n_0,\ldots, n_p \ge 1}$ and a random variable $X$, we write
\begin{equation*}  
U(n_0,\ldots, n_p) \xrightarrow[n_0 \asymp \cdots \asymp n_p \to \infty]{\text{(prob.)}} X
\end{equation*}
to mean that, for every $C < +\infty$ and $\ep,\delta > 0$, there exists $N \in \N$ such that if $n_0, \ldots, n_p$ satisfy
\begin{equation*}  
\forall q,q' \in \{0,\ldots, p\}, \quad N \le n_q \le C n_{q'} ,
\end{equation*}
then
\begin{equation*}  
\P \Ll[ \Ll| U(n_0,\ldots, n_p) - X \Rr| \ge \ep \Rr] \le \delta. 
\end{equation*}
The following theorem covers the case in which the entries of all the matrices are centered. We assume that these entries have a unit variance, which entails no loss of generality, by a simple rescaling. 
\begin{theorem}[centered entries]
Let $\al > 1$, $p \ge 2$ an integer, let $A^{(1)} = (A^{(1)}_{ij})_{1 \le i \le n_0,1 \le j  \le n_1}, \ldots, A^{(p)} = (A^{(p)}_{ij})_{1 \le i \le n_{p-1},1 \le j \le n_p}$ be independent families of i.i.d.\ centered random variables with unit variance, with $A^{(1)}_{ij}, \ldots, A^{(p-1)}_{ij}$ having a finite moment of order $4\al$, and with $A^{(p)}_{ij}$ having a finite moment of order $2\al$. Let 
\begin{equation*}  
P := A^{(1)} \, \cdots \, A^{(p)} \in \R^{n_0 \times n_p}
\end{equation*}
denote the product of these matrices. We have
\begin{equation*}  
n_p^{-1}(n_1 \, \cdots\, n_{p-1})^{-\frac 1 2} \|P\|_{\ell^\infty \to \ell^\infty} \xrightarrow[n_0 \asymp \cdots \asymp n_p \to \infty]{\text{(prob.)}} \sqrt{\frac 2 \pi} .
\end{equation*}
\label{t.centered}
\end{theorem}
\begin{remark}
Assuming that $A_{11}^{(p)}$ has a finite second moment is necessary for the statement of the theorem to make sense. In the special case $p = 2$, Proposition~\ref{p.optim.moment} below shows that a fourth-moment assumption on $A^{(1)}_{11}$ is essentially necessary as well. 
\end{remark}
For all integers $m,n \ge 1$, we write $\1_{m,n}$ to denote the $m$-by-$n$ matrix with all entries equal to $1$. We say that a matrix of the form $\1_{m,n}$ for some integers $m,n \ge 1$ is a $\1$-type matrix. Any matrix with i.i.d.\ and possibly non-centered  entries can be decomposed into the sum of a matrix with i.i.d.\ centered entries and a multiple of a $\1$-type matrix. A product of matrices with i.i.d.\ and possibly non-centered entries can therefore be rewritten as a linear combination of products of matrices with i.i.d.\ centered entries and $\1$-type matrices. We therefore focus on analysing the behavior of such products. So consider a product of $p$ terms, each term being either a matrix with i.i.d.\ centered entries or a $\1$-type matrix. If a single $\1$-type matrix appears in the product, and it is the first term of the product, then the operator norm of the product turns out to be of the same order of magnitude as in the case involving $p$ centered matrices, that is, of order $n^{\frac{p+1}{2}}$ if all matrices are of size $n$ by $n$. 
\begin{theorem}[Products that start with $\1$]
\label{t.1P}
Let $\al > 1$, $p \ge 1$ an integer, let $A^{(1)} = (A^{(1)}_{ij})_{1 \le i \le n_0,1 \le j  \le n_1}, \ldots, A^{(p)} = (A^{(p)}_{ij})_{1 \le i \le n_{p-1},1 \le j \le n_p}$ be independent families of i.i.d.\ centered random variables with unit variance and a finite moment of order $2\al$, and  let 
\begin{equation*}  
P := A^{(1)} \, \cdots \, A^{(p)} \in \R^{n_0 \times n_p}
\end{equation*}
denote the product of these matrices. We have 
\begin{equation}  
\label{e.1P}
n_p^{-1} (n_0 \, \cdots \, n_{p-1})^{-\frac 1 2}\|\1_{m,n_0} P\|_{\ell^\infty \to \ell^\infty} \xrightarrow[n_0 \asymp \cdots \asymp n_p \to \infty]{\text{(prob.)}} \sqrt{\frac 2 \pi}. 
\end{equation}
\end{theorem}
\begin{remark}  
\label{r.triv}
For any matrix $A \in \R^{n_0,m_p}$, the quantity $\|\1_{m,n_0} A\|_{\ell^\infty \to \ell^\infty}$ does not depend on $m$. In particular, the quantity on the left side of \eqref{e.1P} does not depend on $m$, and we therefore do not need to specify how (or whether) $m$ tends to infinity in this result. 
\end{remark}
The next case we consider is that of a product which starts and ends with a $\1$-type matrix, and otherwise only involves matrices with i.i.d.\ centered entries. If the product has a total of $p \ge 3$ terms (and thus involves $p-2$ random matrices), and if all the matrices are $n$-by-$n$, then we find again that the operator norm of the product is of the order of $n^{\frac {p+1}{2}}$. The starting point of the analysis is the elementary observation that, for every matrix $P \in \R^{n_1\times n_3}$, 
\begin{equation*}  
\1_{n_0,n_1} P \1_{n_2, n_3} = \Ll( \sum_{i = 1}^{n_1} \sum_{j = 1}^{n_2} P_{ij} \Rr) \1_{n_0,n_3},
\end{equation*}
so that 
\begin{equation*}  
\|\1_{n_0,n_1} P \1_{n_2, n_3}\|_{\ell^\infty \to \ell^\infty} = n_3 \Ll| \sum_{i = 1}^{n_1} \sum_{j = 1}^{n_2} P_{ij}  \Rr| .
\end{equation*}
The next theorem describes the behavior of this scalar variable. This is the only base case that displays a random limit behavior. 
\begin{theorem}[the case of $\1 P\1$]
\label{t.sum}
Let $\al > 1$, $p \ge 1$ an integer, let $A^{(1)} = (A^{(1)}_{ij})_{1 \le i \le n_0,1 \le j  \le n_1}, \ldots, A^{(p)} = (A^{(p)}_{ij})_{1 \le i \le n_{p-1},1 \le j \le n_p}$ be independent families of i.i.d.\ centered random variables with unit variance and a finite moment of order $2\al$, and  let 
\begin{equation*}  
P := A^{(1)} \, \cdots \, A^{(p)} \in \R^{n_0 \times n_p}
\end{equation*}
denote the product of these matrices. The random variable
\begin{equation*}  
(n_0 \, \cdots \, n_p
)^{-\frac 1 2} \sum_{i = 1}^{n_0} \sum_{j = 1}^{n_p} P_{ij}
\end{equation*}
converges in law to a centered Gaussian random variable of unit variance.
\end{theorem}
The next theorem addresses the case in which a product of matrices with i.i.d.\ centered entries is followed by a $\1$-type matrix. If a total of $p$ matrices is involved (so that there are $p-1$ random matrices in total), and if all matrices are $n$-by-$n$, then we find that the operator norm of the product is slightly larger than in the previous cases, by a factor of $\sqrt{\log n}$. 
\begin{theorem}[Products that end with $\1$]
\label{t.P1}
Let $p \ge 1$ be an integer, let $A^{(1)} = (A^{(1)}_{ij})_{1 \le i \le n_0,1 \le j  \le n_1}, \ldots, A^{(p)} = (A^{(p)}_{ij})_{1 \le i \le n_{p-1},1 \le j \le n_p}$ be independent families of i.i.d.\ centered random variables with unit variance and a finite fourth moment, and let 
\begin{equation*}  
P := A^{(1)} \, \cdots \, A^{(p)} \in \R^{n_0 \times n_p}
\end{equation*}
denote the product of these matrices. We have 
\begin{equation}  
\label{e.P1}
m^{-1} (2 n_1 \, \cdots \, n_{p}  \log n_0)^{-\frac 1 2} \|P\1_{n_p,m}\|_{\ell^\infty \to \ell^\infty} \xrightarrow[n_0 \asymp \cdots \asymp n_{p} \to \infty]{\text{(prob.)}} 1.
\end{equation}
\end{theorem}
\begin{remark}  
For any matrix $A \in \R^{n_0,n_p}$, the quantity $m^{-1} \|A \1_{n_p,m}\|_{\ell^\infty \to \ell^\infty}$ does not depend on $m$. In particular, the quantity on the left side of \eqref{e.P1} does not depend on $m$, and we therefore do not need to specify how (or whether) $m$ diverges to infinity in this result. 
\end{remark}
We now explain how to reduce all other cases to those already covered by the results above. First, notice that if two $\1$-type matrices appear consecutively in a product, then they can be replaced by a single $\1$-type matrix using the elementary identity
\begin{equation}  
\label{e.elentary.11}
\1_{n_0,n_1} \, \1_{n_1, n_2} = n_1 \1_{n_0,n_2}.
\end{equation}
Finally, when a $\1$-type matrix appears neither at the beginning nor at the end of the product, the operator norm of the product can be decomposed according to the following observation.
\begin{lemma}[Decomposition of the operator norm]
\label{l.decomp}
For all matrices $A \in \R^{n_0,n_1}$ and $B \in \R^{n_2,n_3}$, we have
\begin{equation}  
\label{e.decomp1}
\|A \1_{n_1,n_2} B\|_{\ell^\infty \to \ell^\infty} = n_2^{-1} \|A\1_{n_1,n_2}\|_{\ell^\infty \to \ell^\infty} \, \|\1_{n_2,n_2} B\|_{\ell^\infty \to \ell^\infty}.
\end{equation}
\end{lemma}
\begin{proof}
Writing $\1_{n_2} := (1,\ldots, 1) \in \R^{n_2}$, we have for every $x \in \R^{n_2}$ that 
\begin{equation*}  
\1_{n_1,n_2} x = \1_{n_2} |x|_{1}, \quad \text{ so } \quad |A \1_{n_1,n_2} x|_{\infty} = |A \1_{n_2}|_\infty |x|_1.
\end{equation*}
In particular,
\begin{equation*}  
\|A\1_{n_1,n_2}\|_{\ell^\infty \to \ell^\infty} = n_2 |A \1_{n_2}|_\infty
\end{equation*}
and
\begin{multline*}  
\|A \1_{n_1,n_2} B\|_{\ell^\infty \to \ell^\infty} = \sup_{|x|_\infty \le 1} |A\1_{n_1, n_2} Bx|_{\infty} 
\\
= |A \1_{n_2}| \sup_{|x|_\infty \le 1} |Bx|_1 = n_2^{-1}\|A\1_{n_1,n_2}\|_{\ell^\infty \to \ell^\infty} \sup_{|x|_\infty \le 1} |Bx|_1.
\end{multline*}
We obtain the result after observing that $|\1_{n_2,n_2} Bx|_{\infty} = |Bx|_{1}$. 
\end{proof}
Using Lemma~\ref{l.decomp} iteratively, we can therefore reduce our analysis to the case of products involving only matrices with i.i.d.\ centered entries, except possibly in the first and last positions. Those cases are covered by Theorems~\ref{t.centered}, \ref{t.1P}, \ref{t.sum}, and \ref{t.P1}.

The rest of this paper is organized as follows. The short Section~\ref{s.one.mat} provides the proof of Proposition~\ref{p.one.matrix}. Section~\ref{s.centered} focuses on the proof of Theorem~\ref{t.centered}. One important ingredient is Proposition~\ref{p.gaussian}, which gives a quantitative central limit theorem for sums of independent (but not necessarily identically distributed) random variables, and is phrased in terms of the Wasserstein $L^1$ distance. Other ingredients include upper bounds on the moments of the entries of a product of random matrices with i.i.d.\ centered entries, and concentration estimates. All these results will also be used in subsequent sections. Section~\ref{s.sum} concerns the proof of Theorem~\ref{t.sum}. Further moment and concentration estimates, concerning column sums of the coefficients of the product matrix, are derived and also used in the sequel. With all these tools in place, we can readily show Theorem~\ref{t.1P} in Section~\ref{s.1P}. The final Section~\ref{s.P1} concerns the proof of Theorem~\ref{t.P1}. This result has a different nature from the previous cases, as the maximum in the expression for the operator norm in \eqref{e.def.norm} is responsible for the additional logarithmic divergence. The proof of this involves a suitable truncation, as well as an argument in the spirit of moderate deviations for the row sums of the product matrix. 

%
%
%
%
%
%
\section{The case of one matrix}
\label{s.one.mat}

In this short section, we provide the elementary proof of Proposition~\ref{p.one.matrix}. 
\begin{proof}[Proof of Proposition~\ref{p.one.matrix}]
By Rosenthal's inequality (see \cite[Theorem~3]{rosenthal1970subspaces} or \cite[Theorem~15.11]{boucheron2013concentration}), there exists a constant $C < + \infty$ such that 
\begin{multline*}  
\E \Ll[ \Ll| \sum_{j = 1}^n (|A_{ij}| - \E|A_{ij}|) \Rr|^{2\al} \Rr] 
\\
\le C \sum_{j = 1}^n \E \Ll[ \Ll| |A_{ij}| - \E|A_{ij}| \Rr|^{2\al}  \Rr] + C \Ll( \sum_{j = 1}^n \E \Ll[ \Ll( |A_{ij}| - \E|A_{ij}| \Rr)^2  \Rr]  \Rr)^\al,
\end{multline*}
and up to a redefinition of the constant $C$, we can bound the latter by
\begin{equation*}  
C n^\al \E \Ll[ |A_{11}|^{2\al} \Rr] .
\end{equation*}
By Chebyshev's inequality, we thus have for every $i \in \{1,\ldots, m\}$  and $\ep > 0$ that 
\begin{equation*}  
\P \Ll[ \Ll|\sum_{j = 1}^n \Ll(|A_{ij}|- \E \Ll[ |A_{ij}| \Rr]\Rr)\Rr|  \ge \ep n \Rr] \le C \ep^{-2\al} n^{-\al} \E \Ll[ |A_{11}|^{2\al} \Rr] ,
\end{equation*}
By a union bound, the probability of the event
\begin{equation*}  
\mbox{there exists $i \in \{1,\ldots, m\}$ such that } \ \Ll|\sum_{j = 1}^n \Ll(|A_{ij}|- \E \Ll[ |A_{ij}| \Rr]\Rr)\Rr|  \ge \ep n 
\end{equation*}
is bounded by $C \ep^{-2\al} m n^{-\al} \E \Ll[ |A_{11}|^{2\al} \Rr]$. Recalling that we enforce $m = O(n)$, we obtain the result.
\end{proof}

%
%
%
%
%
%
\section{Products of centered matrices}
\label{s.centered}

The goal of this section is to prove Theorem~\ref{t.centered}, which describes the asymptotic behavior of products of independent matrices with i.i.d.\ centered entries. Some of the intermediate results will also be used in the next sections to handle the other cases.

One key ingredient of the proof of Theorems~\ref{t.centered}, \ref{t.1P} and \ref{t.sum} is the following quantitative version of the central limit theorem. The result we present here is similar to that of \cite{rio}, although only i.i.d.\ summands were considered there, while we cover the case of general independent summands here. A similar result can also be deduced from the non-uniform bound of \cite{bikjalis}, albeit with a somewhat more complicated proof and with a non-explicit constant in place of the factor of $4$ in \eqref{e.gaussian}. 

\begin{proposition}[Gaussian approximation]
\label{p.gaussian}	
Let $\al \in \Ll[2, 3\Rr]$, let $(B_j)_{j \ge 1}$ be a sequence of independent centered random variables with finite moment of order~$\al$ such that
$
\sum_{j = 1}^n \E[B_j^2] = 1,
$
and let $\mcl N$ be a standard Gaussian random variable.
For every $1$-Lipschitz function $h : \R \to \R$, we have
\begin{equation}  
\label{e.gaussian}
\bigg|\E\bigg[h\bigg(\sum_{j = 1}^n B_j\bigg)\bigg] - \E[h(\mcl N)]\bigg| \le 4 \sum_{j = 1}^n \E[|B_j|^{\al}]. 
\end{equation}
\end{proposition}
\begin{remark}  
Denoting by $\mu_n$ the law of $\sum_{j =1}^n B_j$, by $\gamma$ the standard one-dimensional Gaussian distribution, and by $W_1(\mu_n, \gamma)$ the Wasserstein $L^1$ distance between $\mu_n$ and $\gamma$, we can rephrase the conclusion of Proposition~\ref{p.gaussian} as
\begin{equation*}  
W_1 ( \mu_n, \gamma) \le  4 \sum_{j = 1}^n \E[|B_j|^{\al}].
\end{equation*}

\end{remark}
\begin{proof}[Proof of Proposition~\ref{p.gaussian}]
We make use of Stein's method of normal approximation, in the spirit of \cite[Subsection~2.3.1]{chen2011normal}. We write 
\begin{equation*}  
X := \sum_{j = 1}^n B_j.
\end{equation*}
As explained in \cite[Lemma~2.4]{chen2011normal}, in order to show the result, it suffices to show that for every smooth function $f : \R \to \R$ satisfying
\begin{equation}
\label{e.stein.bound.f}
\|f\|_{L^\infty} \le 2, \qquad \|f'\|_{L^\infty} \le \sqrt{\frac{2}{\pi}}, \quad \text{ and } \quad \|f''\|_{L^\infty} \le 2,
\end{equation}
we have
\begin{equation}
\label{e.stein.goal}
\Ll| \E[ f'(X) - X f(X) ] \Rr| \le 4 \sum_{j = 1}^n \E[|B_j|^{\al}]. 
\end{equation}

We start with some preliminary observations. For every sufficiently regular and moderately growing function $g : \R \to \R$ and $j \in \{1,\ldots, n\}$, we can write
\begin{equation*}  
\E [ B_j g(B_j) ] = \E [ B_j (g(B_j)-g(0)) ] = \E \Ll[B_j \int_0^{B_j} g'(t) \, \d t \Rr] .
\end{equation*}
Denoting 
\begin{equation*}  
K_j(t) := 
\begin{cases}  
\E[B_j \1_{\{ t \le B_j \}}] & \text{ if } t \ge 0, \\
-\E[B_j \1_{\{ B_j \le t \}}] & \text{ if } t < 0, 
\end{cases}
\end{equation*}
we can appeal to Fubini's theorem to obtain that 
\begin{equation}
\label{e.Stein.ibp}
\E[B_j g(B_j)] = \int g'(t) K_j(t) \, \d t.
\end{equation}
In particular,
\begin{equation}
\label{e.Kj.moments}
\int K_j(t) \, \d t = \E[B_j^2],  \qquad \int |t|^{\al-2} \, K_j(t) \, \d t = (\al-1)^{-1} \E[|B_j|^{\al}],
\end{equation}
and the definition of $K_j$ makes it clear that it takes non-negative values. 

For every $j \in \{1,\ldots, n\}$, we write 
\begin{equation*}  
X^{(j)} := X - B_j.
\end{equation*}
Since $X^{(j)}$ is independent of $B_j$, we can use \eqref{e.Stein.ibp} to write
\begin{align}  
\label{e.rewrite.xf}
\E \Ll[ X f(X) \Rr] = \sum_{j = 1}^n \E [B_j f(X^{(j)} + B_j)]
= \sum_{j = 1}^n \int \E[f'(X^{(j)}+ t)] K_j(t) \, \d t.
\end{align}
The first relation in \eqref{e.Kj.moments} implies that 
\begin{equation*}  
\sum_{j = 1}^n \int K_j(t) \, \d t = 1.
\end{equation*}
We therefore define, for each $j \in \{1,\ldots, n\}$, a new random variable $B_j^*$ with density proportional to $K_j(t) \, \d t$, as well as $J$ such that 
\begin{equation*}  
\P[J = j] = \int K_j(t) \, \d t.
\end{equation*}
We take these random variables to be independent with each other as well as with the other sources of randomness. These random variables allow us to rewrite \eqref{e.rewrite.xf} as 
\begin{equation*}  
\E \Ll[ X f(X) \Rr] = \E \Ll[ f'(X^{(J)} + B_J^*) \Rr] ,
\end{equation*}
so that 
\begin{equation}
\label{e.diff.f'}
\E \Ll[ X f(X) -f'(X)\Rr] = \E \Ll[ f'(X^{(J)} + B_J^*) - f'(X^{(J)} + B_J) \Rr] .
\end{equation}
The identities in \eqref{e.Kj.moments} tell us that
\begin{equation*}  
\E[|B_j^*|^{\al-2}] = \frac{\E[|B_j|^\al]}{(\al-1) \E[B_j^2]},
\end{equation*}
and in particular,
\begin{equation}  
\label{e.moment.Bj*}
\E[|B_J^*|^{\al-2}] = \frac 1 {\al-1} \sum_{j = 1}^n \E[|B_j|^\al].
\end{equation}
Recalling  the bounds on $f'$ and $f''$ from \eqref{e.stein.bound.f}, we see that for every $x,y \in \R$,
\begin{equation*}  
\Ll| f'(y) - f'(x) \Rr| \le  \min \Ll(\frac{2\sqrt{2}}{\sqrt{\pi}}, \ 2|y-x| \Rr) \le 2|y-x|^{\al-2}.
\end{equation*}
Combining this with \eqref{e.diff.f'}, we obtain that 
\begin{equation*}  
\Ll| \E \Ll[ X f(X) -f'(X)\Rr] \Rr| \le 2 \E [|B_J^* - B_J|^{\al-2}] \le 2 \E[|B_J^*|^{\al-2}] + 2 \E[|B_J|^{\al-2}].
\end{equation*}
By H\"older's inequality, we can estimate the last term as
\begin{equation*}  
\E[|B_J|^{\al-2}] = \sum_{j = 1}^n \E[B_j^2] \E[|B_j|^{\al-2}] \le \sum_{j = 1}^n \E[|B_j|^{\al}]. 
\end{equation*}
Combining the last two displays with \eqref{e.moment.Bj*}, we thus obtain \eqref{e.stein.goal}, and thereby complete the proof.
\end{proof}

We next present an estimate on the moments of each entry of the product matrix. We use the shorthand $\al \vee 1 := \max(\al,1)$. 
\begin{proposition}[Moment estimate for matrix entries]
\label{p.moment.entries}
For every $\al \ge \frac 1 2$ and integer $p \ge 1$, there exists a constant $C < +\infty$ such that the following holds. Let $A^{(1)} = (A^{(1)}_{ij})_{1 \le i \le n_0,1 \le j  \le n_1}, \ldots, A^{(p)} = (A^{(p)}_{ij})_{1 \le i \le n_{p-1},1 \le j \le n_p}$ be independent families of i.i.d.\ centered random variables with finite moment of order $2\al$, and let 
\begin{equation*}  
P := A^{(1)} \, \cdots \, A^{(p)} \in \R^{n_0 \times n_p}
\end{equation*}
denote the product of these matrices. For every $i \in \{1,\ldots, n_0\}$ and $j \in \{1,\ldots, n_p\}$, we have
\begin{equation*}  
\E \Ll[ |P_{ij}|^{2\al} \Rr] \le C (n_1 \, \cdots \, n_{p-1})^{\al \vee 1} \, \E[|A^{(1)}_{11}|^{2\al}] \, \cdots \, \E[|A^{(p)}_{11}|^{2\al}].
\end{equation*}
For $\al = 1$, we in fact have the identity
\begin{equation}  
\label{e.var.identity}
\E \Ll[ |P_{ij}|^{2} \Rr] = n_1 \, \cdots \, n_{p-1} \, \E[|A^{(1)}_{11}|^{2}] \, \cdots \, \E[|A^{(p)}_{11}|^{2}].
\end{equation}
\end{proposition}

\begin{proof}
We prove the statement by induction over $p$. The case $p = 1$ is immediate. Now suppose that $p \ge 2$, and denote by $P'$ the product of the first $(p-1)$ matrices, so that $P = P' A^{(p)}$ and
\begin{equation*}  
P_{ik} = \sum_{j = 1}^{n_{p-1}} P'_{ij} A^{(p)}_{jk}.
\end{equation*}
Let us denote by $\E_p$ the expectation with respect to $A^{(p)}$ only, keeping the other variables appearing in $P'$ fixed. 
For $\al \in \Ll[ \frac 1 2, 1 \Rr]$, we appeal to the von Bahr-Esseen inequality \cite{von1965inequalities} to write that 
\begin{equation}  
\label{e.moment.vbe}
\E_p \Ll[ |P_{ik}|^{2\al} \Rr] \le 2 \sum_{j = 1}^{n_{p-1}} |P_{ij}'|^{2\al} \E \Ll[ |A^{(p)}_{jk}|^{2\al} \Rr] .
\end{equation}
Taking the expectation with respect to all randomness and applying the induction hypothesis allows us to conclude. In the case $\al = 1$, we can in fact remove the factor of $2$ and write an equality in \eqref{e.moment.vbe}, and we obtain the result in the same way. 
For $\al \ge 1$, we can appeal instead to 
 Rosenthal's inequality (see \cite[Theorem~3]{rosenthal1970subspaces} or \cite[Theorem~15.11]{boucheron2013concentration}) to obtain that
\begin{equation*}  
\E_p \Ll[ |P_{ik}|^{2\al} \Rr] \le C \sum_{j = 1}^{n_{p-1}} |P_{ij}'|^{2\al} \E \Ll[ |A^{(p)}_{jk}|^{2\al} \Rr]  + C \Ll( \sum_{j = 1}^{n_{p-1}} |P_{ij}'|^{2} \E \Ll[ |A^{(p)}_{jk}|^{2} \Rr] \Rr)^\al.
\end{equation*}
In the last term, we can bring the exponent $\al$ inside the sum at the cost of a multiplicative factor of $n_{p-1}^{\al-1}$. Taking the expectation and using Jensen's inequality to compare moments, we obtain that 
\begin{align*}  
\E \Ll[ |P_{ik}|^{2\al} \Rr] 
& \le  C n_{p-1}^\al \E\Ll[|P_{11}'|^{2\al}\Rr] \E \Ll[ |A^{(p)}_{11}|^{2\al} \Rr],
\end{align*}
and an application of the induction hypothesis concludes the proof.
\end{proof}

The next proposition states a concentration estimate for a quantity which is closely related to the $\ell^\infty \to \ell^\infty$ operator norm of a product of two matrices. For the purposes of this section, we only need to consider the case where the coefficients $(a_{jk})$ appearing there do not in fact depend on $k$, but the more general version entails no additional difficulty and will be useful later on. 
\begin{lemma}
\label{l.concentration}
Let $\alpha \ge 1$, let $(a_{jk})_{j,k \ge 1}$ be real numbers, let $(B_{jk})_{j,k \ge 1}$ be a sequence of independent random variables with finite moment of order~$2\alpha$, and let
\begin{equation*}  
X := \sum_{k = 1}^n \Ll| \sum_{j = 1}^m a_{jk} B_{jk} \Rr|  .
\end{equation*}
We have
\begin{equation*}  
\E \Ll[ \Ll| X - \E[X]  \Rr|^{2\alpha}  \Rr] \le (84 \al)^\al\,  \E \Ll[ \Ll( \sum_{j = 1}^m \sum_{k = 1}^n a_{jk}^2 (B_{jk} - \E[B_{jk}])^2 \Rr)^\al  \Rr] .
\end{equation*}
\end{lemma}
\begin{proof}
Let $(B'_{jk})$ be an independent copy of $(B_{jk})$, and for each $j,k \ge 1$, let $X_{jk}$ be the random variable obtained by substituting $B_{jk}$ with $B'_{jk}$ in the definition of $X$, and let
\begin{equation*}  
V := \sum_{j = 1}^m \sum_{k = 1}^n (X - X_{jk})^2.
\end{equation*}
By the Efron-Stein inequality for moments of order $2\alpha$ (see \cite[Theorem~15.5]{boucheron2013concentration}), we have
\begin{equation*}  
\E \Ll[ \Ll| X - \E[X] \Rr|^{2\al}  \Rr] \le (21 \al)^\al \, \E \Ll[ V^\al \Rr] .
\end{equation*}
We next observe that 
\begin{align*}  
V & \le \sum_{j,k} a_{jk}^2 (B_{jk}  - B_{jk}')^2 
\\
& \le 2 \sum_{j,k} a_{jk}^2 \Ll( (B_{jk} - \E[B_{jk}])^2 + (B'_{jk} - \E[B_{jk}])^2 \Rr) .
\end{align*}
The result then follows using that $|x+y|^\al \le 2^{\al-1}(|x|^\al + |y|^\al)$. 
\end{proof}
The explicit formula for the $\ell^\infty\to \ell^\infty$ operator norm of a matrix involves the calculation of the $\ell^1$ norm of each row. The next lemma and proposition will help us to pin down the variance of each of these terms (with the choice of $\al > 1$ and $\ga = 2$), even if we condition on all matrices except the last one. (The choice of $\al = 1$ and $\ga > 2$ will also be used to control the fluctuations of an error term.)
\begin{lemma}
\label{l.concentration.2}
For each $\al \ge 1$ and $\ga \ge 2$, there exists a constant $C < + \infty$ such that the following holds. Let $(a_{jk})_{j,k \ge 1}$ be real numbers, let $(B_{jk})_{j,k \ge 1}$ be a sequence of independent centered random variables with finite moment of order~$2\alpha\ga$, and let
\begin{equation*}  
X := \sum_{k = 1}^n \Ll| \sum_{j = 1}^m a_{jk} B_{jk} \Rr|^\ga  .
\end{equation*}
We have
\begin{align*}
& \E \Ll[ |X - \E[X]|^{2\al} \Rr] 
\le C m^{\al\ga-1} n^{\al-1} \sum_{j = 1}^m \sum_{k = 1}^n|a_{jk}|^{2\al\ga} \E[|B_{jk}|^{2\al\ga}].
\end{align*}
\end{lemma}
\begin{proof}
We start with the following elementary observation: for each $a \ge 1$, there exists $C < +\infty$ such that for every $x,y \in \R$, 
\begin{equation}  
\label{e.expansion.x.a}
\Ll||x+y|^a - |x|^a \Rr| \le C |y|^a + C |x|^{a-1} |y|.
\end{equation}
To see this, we write
\begin{align*}  
\Ll||x+y|^a - |x|^a\Rr| & = \Ll|\int_0^y \frac{\d}{\d t} \Ll( |x+t|^a \Rr) \, \d t \Rr|
\\
& \le C \Ll|\int_0^y \Ll(|x|^{a-1} + |t|^{a-1}\Rr)\, \d t\Rr|
\\
& \le C |x|^{a-1} |y| + C |y|^{a}.
\end{align*}
We define $(B'_{jk})$, $X_{jk}$, and $V$ as in the proof of Lemma~\ref{l.concentration}, and using~\eqref{e.expansion.x.a}, we have for each fixed $j$ and $k$ that
\begin{align}  
X_{jk} - X & = \Ll| a_{jk}(B'_{jk} - B_{jk}) + \sum_{i = 1}^m a_{ik} B_{ik} \Rr|^\ga -  \Ll|\sum_{i = 1}^m a_{ik} B_{ik} \Rr|^\ga
\notag
\\
\label{e.bineqj}
& \le C |a_{jk}(B'_{jk} - B_{jk})|^\ga + C |a_{jk}(B'_{jk}-B_{jk})| \, \Ll|\sum_{i = 1}^m a_{ik} B_{ik} \Rr|^{\ga - 1}.
\end{align}
We appeal to Rosenthal's inequality to write
\begin{multline*}  
\E \Ll[ \Ll| \sum_{i \neq j} a_{ik} B_{ik} \Rr|^{2\al(\ga-1)}  \Rr] \le C \sum_{i \neq j} |a_{ik}|^{2\al(\ga-1)} \E [|B_{ik}|^{2\al(\ga-1)}] 
\\+ C \Ll(\sum_{i \neq j} |a_{ik}|^2 \E[|B_{ik}|^2]\Rr)^{\al(\ga-1)} .
\end{multline*}
By isolating the summand indexed by $j$ in the sum at the end of~\eqref{e.bineqj} and using independence, we thus obtain that
\begin{multline}  
\label{e.concentr.2.2}
\E \Ll[ \Ll| X_{jk} - X \Rr| ^{2\al} \Rr] \le C |a_{jk}|^{2\al \ga} \E[|B_{jk}|^{2\al\ga}] 
\\
+ C |a_{jk}|^{2\al} \E[|B_{jk}|^{2\al}]  \bigg\{\sum_{i = 1}^m |a_{ik}|^{2\al(\ga-1)} \E [|B_{ik}|^{2\al(\ga-1)}]  +  \Ll(\sum_{i = 1}^m |a_{ik}|^{2} \E [|B_{ik}|^{2}]\Rr)^{\al(\ga-1)}  \bigg\}.
\end{multline}
Using that, for any $Z_i \ge 0$ and $\be \ge 1$, we have $(\sum_{i = 1}^m Z_i)^\be \le m^{\be-1} \sum_{i = 1}^m Z_i^\be$, we can simplify the bound \eqref{e.concentr.2.2} into
\begin{multline}  
\label{e.concentr.2.3}
\E \Ll[ \Ll| X_{jk} - X \Rr| ^{2\al} \Rr] \le C |a_{jk}|^{2\al\ga} \E[|B_{jk}|^{2\al\ga}] 
\\
+ C m^{\al(\ga-1)-1} |a_{jk}|^{2\al} \E[|B_{jk}|^{2\al}]  \sum_{i = 1}^m |a_{ik}|^{2\al(\ga-1)} \E [|B_{ik}|^{2\al(\ga-1)}].
\end{multline}
Since
\begin{equation*}  
\E \Ll[\Ll(\sum_{j = 1}^m \sum_{k = 1}^n (X - X_{jk})^2 \Rr)^\al \Rr] \le (mn)^{\al-1} \sum_{j = 1}^m \sum_{k = 1}^n \E \Ll[ |X-X_{jk}|^{2\al} \Rr] ,
\end{equation*}
we can now appeal to the Efron-Stein inequality for moments of order $2\alpha$ (see \cite[Theorem~15.5]{boucheron2013concentration}) and sum the estimate in \eqref{e.concentr.2.3} over all $j,k$ to obtain that
\begin{multline*}
 \E \Ll[ |X - \E[X]|^{2\al} \Rr] 
\le C (mn)^{\al-1} \sum_{j = 1}^m \sum_{k = 1}^n|a_{jk}|^{2\al\ga} \E[|B_{jk}|^{2\al\ga}]
\\
 + C m^{\al\ga-2}n^{\al-1}  \sum_{k = 1}^n \Ll( \sum_{j = 1}^m |a_{jk}|^{2\al} \E [|B_{jk}|^{2\al}] \Rr)\Ll( \sum_{j = 1}^m |a_{jk}|^{2\al(\ga-1)} \E [|B_{jk}|^{2\al(\ga-1)}] \Rr)
.
\end{multline*}
Jensen's inequality ensures that 
\begin{equation*}  
\sum_{j = 1}^m |a_{jk}|^{2\al} \E [|B_{jk}|^{2\al}] \le m^{1-\frac 1 \ga} \Ll( \sum_{j = 1}^m |a_{jk}|^{2\al\ga} \E [|B_{jk}|^{2\al\ga}] \Rr)^{\frac 1 \ga},
\end{equation*}
and similarly,
\begin{equation*}  
\sum_{j = 1}^m |a_{jk}|^{2\al(\ga-1)} \E [|B_{jk}|^{2\al(\ga-1)}] \le m^{1-\frac {\ga-1} \ga} \Ll( \sum_{j = 1}^m |a_{jk}|^{2\al\ga} \E [|B_{jk}|^{2\al\ga}] \Rr)^{\frac {\ga-1} \ga}.
\end{equation*}
We obtain the result by combining the three previous displays. 
\end{proof}
We now upgrade the previous concentration estimate for the variance so that it applies to the entries of a product of random matrices.
\begin{proposition}[Concentration estimate for the variance]
\label{p.concentration.2}
For every $\al \ge 1$ and integer $p \ge 1$, there exists a constant $C < +\infty$ such that the following holds. Let $A^{(1)} = (A^{(1)}_{ij})_{1 \le i \le n_0,1 \le j  \le n_1}, \ldots, A^{(p)} = (A^{(p)}_{ij})_{1 \le i \le n_{p-1},1 \le j \le n_p}$ be independent families of i.i.d.\ centered random variables with finite moment of order $4\al$, and let 
\begin{equation*}  
P := A^{(1)} \, \cdots \, A^{(p)} \in \R^{n_0 \times n_p}
\end{equation*}
denote the product of these matrices. For every $i \in \{1,\ldots, n_0\}$, we have 
\begin{multline*}  
\E \Ll[ \Ll| \sum_{j = 1}^{n_p} \Ll(|P_{ij}|^2 - \E [|P_{ij}|^2] \Rr)\Rr| ^{2\al}  \Rr] 
\\
\le C (n_1 \, \cdots \, n_{p})^{2\al} (n_1^{-\al} + \cdots + n_p^{-\al}) \, \E[|A^{(1)}_{11}|^{4\al}] \, \cdots \, \E[|A^{(p)}_{11}|^{4\al}]. 
\end{multline*}
\end{proposition}
\begin{proof}
For $p = 1$, Rosenthal's inequality yields that 
\begin{multline*}  
\E \Ll[ \Ll| \sum_{j = 1}^{n_1} \Ll(|P_{ij}|^2 - \E [P_{ij}^2] \Rr)\Rr| ^{2\al}  \Rr] 
\le C \sum_{j = 1}^{n_1} \E[|\Ll(|P_{ij}|^2 - \E [P_{ij}^2] \Rr)|^{2\al}] 
\\ + C \Ll( \sum_{j = 1}^{n_1} \E \Ll[ \Ll(|P_{ij}|^2 - \E [P_{ij}^2] \Rr)^2  \Rr] \Rr)^{\al},
\end{multline*}
which implies the desired bound. 

We now proceed by induction, assume that $p \ge 2$, and use the same notation as in the proof of Proposition~\ref{p.moment.entries}, that is, we denote by $P'$ the product of the first $p-1$ matrices, and by $\E_p$ the expectation with respect to the matrix $A^{(p)}$ only, keeping the other variables fixed. We appeal to Lemma~\ref{l.concentration.2} to write that
\begin{align*}  
 \E_p \Ll[ \Ll| \sum_{j = 1}^{n_p} \Ll(|P_{ij}|^2 - \E_p [|P_{ij}|^2] \Rr)\Rr| ^{2\al}  \Rr] 
& \le C n_{p-1}^{2\al-1} n_p^{\al-1} \sum_{j = 1}^{n_{p-1}} \sum_{k = 1}^{n_p} |P'_{ij}|^{4\al} \E[|A^{(p)}_{jk}|^{4\al}].
\end{align*}
Taking the expectation with respect to all randomness and using the moment bound from Proposition~\ref{p.moment.entries} yields that
\begin{multline*}  
\E \Ll[ \Ll| \sum_{j = 1}^{n_p} \Ll(|P_{ij}|^2 - \E_p [|P_{ij}|^2] \Rr)\Rr| ^{2\al}  \Rr] 
\\
 \le C (n_1 \, \cdots \, n_{p-1})^{2\al} n_p^\al \E[|A^{(1)}_{11}|^{4\al}] \, \cdots \, \E[|A^{(p)}_{11}|^{4\al}]. 
\end{multline*}
We also observe that, for every $i \le n_0$ and $k \le n_p$,
\begin{equation*}  
\E_p \Ll[ |P_{ik}|^2 \Rr] = \E_p \Ll[ \Ll| \sum_{j = 1}^{n_{p-1}} P'_{ij} A^{(p)}_{jk} \Rr| ^2 \Rr] = \sum_{j = 1}^{n_{p-1}} |P'_{ij}|^2 \E[|A_{11}^{(p)}|^2].
\end{equation*}
The induction hypothesis guarantees that 
\begin{multline*}  
\E \Ll[ \Ll| \sum_{j = 1}^{n_{p-1}} \big(|P'_{ij}|^2 - \E[|P'_{ij}|^2]\big) \Rr| ^{2\al}  \Rr] 
\\
\le C (n_1 \, \cdots \, n_{p-1})^{2\al} (n_1^{-\al} + \cdots + n_{p-1}^{-\al}) \E[|A^{(1)}_{11}|^{4\al}] \, \cdots \, \E[|A^{(p-1)}_{11}|^{4\al}]. 
\end{multline*}
Combining the three previous displays with the identity \eqref{e.var.identity}, we obtain the desired result.
\end{proof}
In order to control some error term, it would also be convenient to have a variant of Proposition~\ref{p.concentration.2} at our disposal in which the exponent $2$ on $|P_{ij}|$ is replaced by an exponent strictly larger than $2$. For our purposes, we can content ourselves with the following slightly weaker statement (in which we write $x_+ := \max(x,0)$ to denote the positive part).
\begin{lemma}
\label{l.fluct.alpha}
For every $\al \ge 1$, $C_0 < +\infty$, and integer $p \ge 1$, there exists a constant $C < +\infty$ such that the following holds for all integers $n = n_0, n_1, \ldots, n_p$ satisfying the constraint
\begin{equation*}  
\forall q,q' \in \{0,\ldots, p\}, \quad n_q \le C_0 n_{q'}.
\end{equation*}
Let $A^{(1)} = (A^{(1)}_{ij})_{1 \le i \le n_0,1 \le j  \le n_1}, \ldots, A^{(p)} = (A^{(p)}_{ij})_{1 \le i \le n_{p-1},1 \le j \le n_p}$ be independent families of i.i.d.\ centered random variables with finite moment of order $4\al$, and let 
\begin{equation*}  
P := A^{(1)} \, \cdots \, A^{(p)} \in \R^{n_0 \times n_p}
\end{equation*}
denote the product of these matrices. For every $i \in \{1,\ldots, n_0\}$, we have 
\begin{multline}  
\label{e.fluct.alpha}
\E \Ll[  \Ll(\sum_{j = 1}^{n_p} \big(|P_{ij}|^{2\al} - C n^{(p-1)\al} \, \E[|A^{(1)}_{11}|^{2\al}] \, \cdots \, \E[|A^{(p)}_{11}|^{2\al}] \big)\Rr)_+^2 \Rr] 
\\
\le C n^{2\al(p-1) + 1} \, \E[|A^{(1)}_{11}|^{4\al}] \, \cdots \, \E[|A^{(p)}_{11}|^{4\al}]. 
\end{multline}
\end{lemma}
\begin{proof}
For $p = 1$, independence of the entries of the matrix readily yields that
\begin{align*}  
\E \Ll[ \Ll( \sum_{j = 1}^{n_1} |P_{ij}|^{2\al} - \E[|P_{ij}|^{2\al}] \Rr)^2  \Rr] 
& = \sum_{j = 1}^{n_1} \E \Ll[ \Ll( |P_{ij}|^{2\al} - \E[|P_{ij}|^{2\al}] \Rr)^2  \Rr] 
\\
& \le  n_1 \E \Ll[ |A^{(1)}_{11}|^{4\al} \Rr] ,
\end{align*}
and this implies the announced bound. 

Arguing by induction, we now assume that $p \ge 2$. By Lemma~\ref{l.concentration.2}, we have
\begin{align*}  
 \E_p \Ll[ \Ll| \sum_{j = 1}^{n_p} \Ll(|P_{ij}|^{2\al} - \E_p [|P_{ij}|^{2\al}] \Rr)\Rr| ^{2}  \Rr] 
& \le C n_{p-1}^{2\al-1}  \sum_{j = 1}^{n_{p-1}} \sum_{k = 1}^{n_p} |P'_{ij}|^{4\al} \E[|A^{(p)}_{jk}|^{4\al}].
\end{align*}
Taking the expectation with respect to all randomness and using the moment bound from Proposition~\ref{p.moment.entries} yields that
\begin{equation*}  
\E \Ll[ \Ll| \sum_{j = 1}^{n_p} \Ll(|P_{ij}|^{2\al} - \E_p [|P_{ij}|^{2\al}] \Rr)\Rr| ^{2}  \Rr] 
 \le C n^{2\al(p-1)+1}  \E[|A^{(1)}_{11}|^{4\al}] \, \cdots \, \E[|A^{(p)}_{11}|^{4\al}]. 
\end{equation*}
From the identity just below, we see that the quantity $\E_p \Ll[ |P_{ik}|^{2\al} \Rr]$ does not depend on $k$, and by Rosenthal's inequality, we have
\begin{align*}  
\E_p \Ll[ |P_{ik}|^{2\al} \Rr] 
& = \E_p \Ll[ \Ll| \sum_{j = 1}^{n_{p-1}} P'_{ij} A^{(p)}_{jk} \Rr|^{2\al} \Rr] 
\\
& \le C \sum_{j = 1}^{n_{p-1}} |P'_{ij}|^{2\al}\E[|A^{(p)}_{jk}|^{2\al}] + C \Ll(\sum_{j = 1}^{n_{p-1}} |P'_{ij}|^2 \E[|A^{(p)}_{jk}|^2]\Rr)^{\al}
\\
& \le C n_{p-1}^{\al-1} \E[|A^{(p)}_{11}|^{2\al}]  \sum_{j = 1}^{n_{p-1}} |P'_{ij}|^{2\al}.
\end{align*}
We use the induction hypothesis to estimate the last sum as
\begin{multline*}  
\E \Ll[  \Ll(\sum_{j = 1}^{n_{p-1}} \big(|P'_{ij}|^{2\al} - C n^{(p-2)\al} \, \E[|A^{(1)}_{11}|^{2\al}] \, \cdots \, \E[|A^{(p-1)}_{11}|^{2\al}] \big)\Rr)_+^2 \Rr] 
\\
\le C n^{2\al(p-2)+1}\, \E[|A^{(1)}_{11}|^{4\al}] \, \cdots \, \E[|A^{(p-1)}_{11}|^{4\al}]. 
\end{multline*}
Combining the last three displays yields the desired result. 
\end{proof}

We are now ready to complete the proof of Theorem~\ref{t.centered}.
\begin{proof}[Proof of Theorem~\ref{t.centered}]
Without loss of generality, we can assume that $\alpha \le 3/2$ (so that we can later appeal to Proposition~\ref{p.gaussian} we $\alpha$ replaced by $2\alpha$ there). 
We fix a constant $C_0 < +\infty$ and only consider integers $n_0,\ldots, n_p$ that satisfy
\begin{equation}  
\label{e.fix.magnitudes}
\forall q,q' \in \{0,\ldots, p\}, \quad n_q \le C_0 n_{q'}.
\end{equation}
In some estimates that are only monitored up to a multiplicative constant, we use the shorthand notation $n := n_0$. 
We denote by $\E_p$ the expectation with respect to $A^{(p)}$ only, keeping all the other random variables fixed. For convenience of notation, we use the shorthand $B := A^{(p)}$. 
We denote by $P'$ the product of the first $p-1$ matrices, so that $P = P' B$. We thus aim to estimate
\begin{equation*}  
\|P\|_{\ell^\infty \to \ell^\infty} = \max_{1 \le i \le n_0} \sum_{k = 1}^{n_p} \Ll| \sum_{j = 1}^{n_{p-1}} P'_{ij} B_{jk} \Rr| .
\end{equation*}
Throughout this proof, the constant $C < +\infty$ is allowed to depend on the laws of the underlying random variables (through their moments), and may change from one occurence to another.

By Lemma~\ref{l.concentration} (which here we use in a case where, in the notation of this proposition, the coefficient $a_{jk}$ does not depend on $k$) and Jensen's inequality, we have for every~$i \le n_0$ that
\begin{equation*}  
\E_p \Ll[ \Ll| \sum_{k = 1}^{n_p} \Ll(\Ll| \sum_{j = 1}^{n_{p-1}} P'_{ij} B_{jk} \Rr| - \E_p\Ll| \sum_{j = 1}^{n_{p-1}} P'_{ij} B_{jk} \Rr| \Rr)\Rr|^{2\al} \Rr]
 \le C  n_{p-1}^{\al-1} n_p^\al \sum_{j = 1}^{n_{p-1}} |P'_{ij}|^{2\al} .
\end{equation*}
Proposition~\ref{p.moment.entries} ensures that 
\begin{equation}
\label{e.basic.P'}
\E[|P_{ij}'|^{2\al}] \le C n^{(p-2)\al},
\end{equation}
and thus
\begin{equation*}  
\E \Ll[ \Ll| \sum_{k = 1}^{n_p} \Ll(\Ll| \sum_{j = 1}^{n_{p-1}} P'_{ij} B_{jk} \Rr| - \E_p\Ll| \sum_{j = 1}^{n_{p-1}} P'_{ij} B_{jk} \Rr| \Rr)\Rr|^{2\al} \Rr] \le C n^{p \al }.
\end{equation*}
We therefore obtain that, for every $i \le n_0$,
\begin{equation*}  
\P \Ll[ \Ll| \sum_{k = 1}^{n_p} \Ll(\Ll| \sum_{j = 1}^{n_{p-1}} P'_{ij} B_{jk} \Rr| - \E_p\Ll| \sum_{j = 1}^{n_{p-1}} P'_{ij} B_{jk} \Rr| \Rr)\Rr| \ge \ep n^{\frac{p+1}{2}} \Rr] \le C \ep^{-2\al} n^{-\al}.
\end{equation*}
By a union bound, in order to prove the theorem, it suffices to show that
\begin{equation}  
\label{e.main.reduced}
(n_1\, \cdots \, n_{p-1})^{-\frac 1 2} \max_{1 \le i \le n_0} \E_p \Ll| \sum_{j = 1}^{n_{p-1}} P'_{ij} B_{j1} \Rr|  \xrightarrow{\text{(prob.)}} \sqrt{\frac 2 \pi}.
\end{equation}
Applying Proposition~\ref{p.gaussian} with $B_j$ there replaced by 
\begin{equation*}  
\frac{P'_{ij}}{\Ll(\sum_{j = 1}^{n_{p-1}} |P'_{ij}|^2\Rr)^{\frac 1 2}} B_{j1},
\end{equation*}
which has a finite $\E_p$-moment of order $2\alpha \in [2,3]$, 
we have for every $i \le n_0$ that
\begin{equation*}  
\Ll| \E_p \Ll| \sum_{j = 1}^{n_{p-1}} P'_{ij} B_{j1} \Rr| - \sqrt{\frac 2 \pi} \Ll( \sum_{j = 1}^{n_{p-1}} |P_{ij}'|^2 \Rr) ^\frac 1 2\Rr| 
\le 
4 \Ll( \sum_{j = 1}^{n_{p-1}} |P_{ij}'|^2 \Rr)^{\frac 1 2-\al} \sum_{j = 1}^{n_{p-1}} |P'_{ij}|^{2\al}. 
\end{equation*}
By Proposition~\ref{p.concentration.2}, we have
\begin{equation*}  
\E \Ll[ \Ll| \sum_{j = 1}^{n_{p-1}} \Ll(|P'_{ij}|^2 - \E [|P'_{ij}|^2] \Rr)\Rr| ^{2\al}  \Rr] 
\le C n^{2(p-2)\al+\al},
\end{equation*}
so for every $\ep > 0$,
\begin{equation}  
\label{e.main.no.fluct}
\P \Ll[ \Ll| \sum_{j = 1}^{n_{p-1}} \Ll(|P'_{ij}|^2 - \E [|P'_{ij}|^2] \Rr)\Rr| \ge \ep n^{p-1} \Rr] \le C \ep^{-2\al} n^{-\al}.
\end{equation}
Recalling also from Proposition~\ref{p.moment.entries} that 
\begin{equation*}  
\E[|P_{ij}'|^2] = n_1\, \cdots \, n_{p-2},
\end{equation*}
we obtain \eqref{e.main.reduced} provided that we can assert that, for every $\ep > 0$,
\begin{equation}
\label{e.main.error.to.control}
\P \Ll[ \Ll( \sum_{j = 1}^{n_{p-1}} |P_{ij}'|^2 \Rr)^{\frac 1 2-\al} \sum_{j = 1}^{n_{p-1}} |P'_{ij}|^{2\al} \ge \ep n^{\frac{p-1}{2}} \Rr] = o \Ll( n^{-1} \Rr) .
\end{equation}
Using \eqref{e.main.no.fluct}, we see that it suffices to show that, for every $\ep > 0$,
\begin{equation}  
\label{e.main.end}
\P \Ll[  \sum_{j = 1}^{n_{p-1}} |P'_{ij}|^{2\al} \ge \ep n^{(p-1)\al} \Rr] = o(n^{-1}).
\end{equation}
By Lemma~\ref{l.fluct.alpha}, we have
\begin{equation*}  
\E \Ll[ \Ll( \sum_{j = 1}^{n_{p-1}} \Ll(|P'_{ij}|^{2\al} - C n^{(p-2)\al}\Rr)\Rr)_+^2 \Rr] \le C n^{2(p-2)\al + 1},
\end{equation*}
so by Chebyshev's inequality,
\begin{equation*}  
\P \Ll[ \sum_{j = 1}^{n_{p-1}} \Ll(|P'_{ij}|^{2\al} - C n^{(p-2)\al}\Rr) \ge \ep n^{(p-1)\al} \Rr] \le C \ep^{-2} n^{1-2\al}. 
\end{equation*}
This implies \eqref{e.main.end} and thus completes the proof.
\end{proof}
We now show by way of example that, in the setting of Theorem~\ref{t.centered} and for $p = 2$, the assumption that the entries of $A^{(1)}$ have a finite moment of order $4\al$ for some $\al > 1$ is essentially necessary. 
\begin{proposition}[Optimality of moment assumption]
\label{p.optim.moment}
The following holds for every $\ep > 0$ sufficiently small.  There exist two independent families $A = (A_{ij})_{i,j\le n}$, $B = (B_{jk})_{j,k \le n}$ of i.i.d.\ centered random variables, with $B_{11}$ having finite moments of every order and $A_{11}$ having a finite moment of order $4-10\eps$, such that 
\begin{equation*}  
\lim_{n \to +\infty} \P \Ll[ \|AB\|_{\ell^\infty \to \ell^\infty} \le n^{\frac 3 2 + \ep} \Rr] = 0.
\end{equation*}
\end{proposition}
\begin{proof}
We fix the distribution of $A_{11}$ in such a way that for every $x \ge 1$,
\begin{equation*}  
\P \Ll[ |A_{11}| \ge x \Rr] = x^{-4+9\ep}.
\end{equation*}
This random variable has a finite moment of order $4-10\ep$.
We also choose the distribution of $B_{11}$ to be Gaussian with unit variance. We observe that 
\begin{equation*}  
\P[\exists i,j \le n \ : \ |A_{ij}| \ge n^{\frac 1 2 + \ep}] = 1-\Ll( 1-n^{-(\frac 1 2 + \ep)(4-9\ep)} \Rr)^{n^2} .
\end{equation*}
For $\ep > 0$ sufficiently small, the quantity above tends to $1$ as $n$ tends to infinity.  Let $I,J \le n$ be $A$-measurable random indices such that 
\begin{equation*}  
A_{IJ} = \max_{i,j \le n} A_{ij}.
\end{equation*}
Conditionally on $A$, the random variables 
\begin{equation*}  
\Ll( \sum_{j = 1}^n A_{Ij} B_{jk} \Rr) _{k \le n}
\end{equation*}
are independent Gaussians with the same variance $\sum_{j = 1}^n A_{Ij}^2 \ge A_{IJ}^2$. In particular, by taking $\delta > 0$ sufficiently small $(\de = 1/2$ would do), we can make sure that
\begin{equation*}  
\P \Ll[\Ll| \sum_{j = 1}^n A_{Ij} B_{jk} \Rr| \ge   \de |A_{IJ}| \ \ \Big| \ \ A\Rr] \ge \frac 1 2.
\end{equation*}
By independence, we deduce (using for instance \cite[(2.18)]{mountford2013lyapunov}) that
\begin{equation*}  
\P \Ll[\sum_{k = 1}^n\Ll| \sum_{j = 1}^n A_{Ij} B_{jk} \Rr| \ge  \frac{n \de}{4} |A_{IJ}| \ \ \Big| \ \ A\Rr] \ge 1-e^{-\frac{n}{16}}.
\end{equation*}
Since $|A_{IJ}| \ge n^{\frac 1 2 + \ep}$ with probability tending to $1$, this completes the proof. 
\end{proof}

%
%
%
%
%
%
\section{Products that start and end with 1}
\label{s.sum}

In this section, we prove Theorem~\ref{t.sum}.

\begin{proof}[Proof of Theorem~\ref{t.sum}]
We only consider the case $p \ge 2$, the case $p = 1$ being a direct application of the central limit theorem. Without loss of generality, we assume that $\alpha \le \frac 3 2$. We denote by $P'$ the product of the first $p-1$ matrices, so that $P = P'A^{(p)}$, and write $\E_p$ for the expectation with respect to $A^{(p)}$. We impose the constraint \eqref{e.fix.magnitudes} throughout. Denoting by $\mcl N$ a standard Gaussian random variable, we learn from Proposition~\ref{p.gaussian} that for every $1$-Lipschitz function $h : \R \to \R$, we have
\begin{multline*}  
\Ll| \E_p\Ll[h \Ll( \Ll( n_p\sum_{j= 1}^{n_{p-1}} \Ll(\sum_{i = 1}^{n_0} P'_{ij}\Rr)^2 \Rr)^{-\frac 1 2}  \sum_{i,j,k} P'_{ij} A^{(p)}_{jk}  \Rr) \Rr] - \E[h(\mcl N)]\Rr| 
\\
\le 4 n_p^{1-\al} \Ll( \sum_{j= 1}^{n_{p-1}} \Ll(\sum_{i = 1}^{n_0} P'_{ij}\Rr)^2 \Rr)^{-\al}   \sum_{j =1}^{n_{p-1}} \Ll| \sum_{i = 1}^{n_0} P'_{ij}\Rr|^{2\al} ,
\end{multline*}
where the summation over $i,j,k$ is over $i\le n_0$, $j \le n_{p-1}$, and $k \le n_p$.
The proof will be complete provided that we can show that, for some constant $C< +\infty$ which may depend on the law of the matrix entries,
\begin{equation}  
\label{e.tsum.1}
\E \Ll[ \Ll| \sum_{i = 1}^{n_0} P'_{ij}\Rr|^{2\al} \Rr] \le C (n_0\, \cdots \, n_{p-2})^{\al},
\end{equation}
as well as
\begin{equation}  
\label{e.tsum.2}
\Ll( n_0 \, \cdots \, n_{p-1} \Rr) ^{-1}  \sum_{j= 1}^{n_{p-1}} \Ll(\sum_{i = 1}^{n_0} P'_{ij}\Rr)^2 \xrightarrow[n_0 \asymp \cdots \asymp n_{p-1} \to \infty]{\text{(prob.)}} 1.
\end{equation}
These two results follow from the next two lemmas respectively. 
\end{proof}
\begin{lemma}
\label{l.concentr.1A1.1}
For every $\al \ge 1$, and integer $p \ge 1$, there exists a constant $C < +\infty$ such that the following holds. Let $A^{(1)} = (A^{(1)}_{ij})_{1 \le i \le n_0,1 \le j  \le n_1}, \ldots, A^{(p)} = (A^{(p)}_{ij})_{1 \le i \le n_{p-1},1 \le j \le n_p}$ be independent families of i.i.d.\ centered random variables with unit variance and a finite moment of order $2\al$, and  let 
\begin{equation*}  
P := A^{(1)} \, \cdots \, A^{(p)} \in \R^{n_0 \times n_p}
\end{equation*}
denote the product of these matrices. For every $j \in \{1,\ldots, n_p\}$, we have
\begin{equation*}  
\E \Ll[\Ll| \sum_{i = 1}^{n_0}  P_{ij} \Rr| ^{2\al}\Rr] \le C \E[|A^{(1)}_{11}|^{2\al}] \, \cdots \, \E[|A^{(p)}_{11}|^{2\al}] \, (n_0\, \cdots \, n_{p-1})^{\al}.
\end{equation*}
\end{lemma}
\begin{proof}
For $p = 1$, the result follows directly from Rosenthal's inequality. We now consider the case $p \ge 2$. We use the notation $\E_1$ to denote the expectation with respect to $A^{(1)}$ only, and  denote by $P'$ the product of the $p-1$ last matrices, so that $P = A^{(1)} P'$. For each $k \in \{1,\ldots, n_p\}$, Rosenthal's inequality gives us that
\begin{align*}  
\E_1 \Ll[\Ll| \sum_{i = 1}^{n_0}  P_{ik} \Rr| ^{2\al}\Rr] 
& = \E_1 \Ll[\Ll| \sum_{i = 1}^{n_0} \sum_{j = 1}^{n_1}  A^{(1)}_{ij} P'_{jk} \Rr| ^{2\al}\Rr]
\\
& \le C \sum_{i = 1}^{n_0} \sum_{j = 1}^{n_1} |P'_{jk}|^{2\al}  \E \Ll[ |A^{(1)}_{ij}| ^{2\al} \Rr]
+ C \Ll( \sum_{i = 1}^{n_0}  \sum_{j = 1}^{n_1} |P'_{jk}|^2 \E \Ll[| A^{(1)}_{ij}|^{2} \Rr] \Rr) ^{\al}
\\
& \le C (n_0 n_1)^{\al - 1} \sum_{i = 1}^{n_0} \sum_{j = 1}^{n_1}  |P'_{jk}|^\al \E \Ll[|A^{(1)}_{ij}| ^{2\al} \Rr],
\end{align*}
where we used Jensen's inequality in the last step. Taking the expectation with respect to all randomness and using the moment estimate from Proposition~\ref{p.moment.entries}, we obtain the announced result. 
\end{proof}

\begin{lemma}
\label{l.concentr.1A1.2}
Under the same assumptions as in Lemma~\ref{l.concentr.1A1.1}, and assuming also that $\al \le 2$, there exists a constant $C < +\infty$ depending only on $p$ and $\al$ such that
\begin{multline*}  
\E \Ll[ \Ll| \sum_{j= 1}^{n_p} \Ll(\sum_{i = 1}^{n_0} P_{ij}\Rr)^2 -n_0 \, \cdots \,  n_p  \Rr|^\al  \Rr] 
\\
\le C (n_0 \, \cdots \,  n_p)^\al \Ll( n_1^{1-\al} + \cdots + n_p^{1-\al} \Rr)  \E[|A^{(1)}_{11}|^{2\al}] \, \cdots \, \E[|A^{(p)}_{11}|^{2\al}].
\end{multline*}
\end{lemma}
\begin{proof}
We prove the statement by induction on $p$. 
In the case $p = 1$, the von Bahr-Esseen inequality gives us that
\begin{align*}  
\E \Ll[ \Ll| \sum_{j= 1}^{n_1} \Ll(\sum_{i = 1}^{n_0} P_{ij}\Rr)^2 - \sum_{j = 1}^{n_1} \sum_{i = 1}^{n_0} \E\Ll[|P_{ij}|^2 \Rr]\Rr|^\al  \Rr]
& \le 2 \sum_{j = 1}^{n_1} \E \Ll[\Ll| \Ll( \sum_{i = 1}^{n_0} P_{ij} \Rr) ^2 - \sum_{i = 1}^{n_0} \E[P_{ij}^2] \Rr|^\al \Rr] 
\\
& \le 4 \sum_{j = 1}^{n_1} \E \Ll[ \Ll| \sum_{i = 1}^{n_0} P_{ij} \Rr| ^{2\al} \Rr] 
\\
& \le C n_0^\al n_1 \E \Ll[ |P_{11}|^{2\al} \Rr] ,
\end{align*}
where we also used Lemma~\ref{l.concentr.1A1.1} (i.e.\ Rosenthal's inequality) in the last step. Assuming from now on that $p \ge 2$, we use the same notation as in the proof of Theorem~\ref{t.centered} to write $B = A^{(p)}$ and $P'$ the product of the first $p-1$ matrices, so that $P = P' B$, and we write $\E_p$ to denote the expectation with respect to $B = A^{(p)}$ only. By the von Bahr-Esseen inequality,
\begin{multline*}  
\E_p \Ll[ \Ll| \sum_{k= 1}^{n_p} \Ll(\sum_{i,j} P'_{ij}B_{jk}\Rr)^2 - n_p \E_p \Ll[ \Ll(\sum_{i,j} P'_{ij}B_{jk}\Rr)^2  \Rr] \Rr|^\al  \Rr] 
\\ \le 2 \sum_{k = 1}^{n_p}\E_p \Ll[ \Ll| \Ll(\sum_{i,j} P'_{ij}B_{jk}\Rr)^2 -  \E_p \Ll[ \Ll(\sum_{i,j} P'_{ij}B_{jk}\Rr)^2  \Rr] \Rr|^\al  \Rr] ,
\end{multline*}
where the summation indices $i,j$ range over $i \le n_0$ and $j \le n_{p-1}$. Recalling that $\sum_{j} P'_{ij} B_{jk} = P_{ik}$, we can appeal to  Lemma~\ref{l.concentr.1A1.1}  to obtain that 
\begin{multline*}  
\E \Ll[ \Ll| \sum_{k= 1}^{n_p} \Ll( \sum_{i = 1}^{n_0} P_{ij} \Rr) ^2 - n_p  \sum_{j=1}^{n_{p-1}} \Ll(\sum_{i = 1}^{n_0} P'_{ij}\Rr)^2   \Rr|^\al  \Rr] 
\\
 \le C (n_0 \, \cdots \, n_{p-1})^\al n_p \E[|A^{(1)}_{11}|^{2\al}] \, \cdots \, \E[|A^{(p)}_{11}|^{2\al}] .
\end{multline*}
The induction hypothesis then allows us to conclude.
\end{proof}

%
%
%
%
%
%
\section{Products starting with 1}
\label{s.1P}

In this section, we focus on the proof of Theorem~\ref{t.1P}, which concerns products of random matrices with centered entries preceded by a $\1$-type matrix. We start by considering the case in which a single random matrix is involved. 

\begin{proposition}[The case of $1A$]
\label{p.1A}
Let $A^{(1)} = (A^{(1)}_{ij})_{1 \le i \le n_0,1 \le j  \le n_1}$ be a family of i.i.d.\ centered random variables with unit variance. We have
\begin{equation*}  
n_1^{-1} n_0^{-\frac 1 2} \|\1_{m,n_0} A\|_{\ell^\infty \to \ell^\infty} \xrightarrow[n_0 \asymp n_1 \to \infty]{\text{(prob.)}} \sqrt{\frac 2 \pi} .
\end{equation*}
\end{proposition}
\begin{proof}
We start by noting that
\begin{equation*}  
\|1A\|_{\ell^\infty \to \ell^\infty} = \sum_{j = 1}^{n_1} \Ll| \sum_{i = 1}^{n_0} A_{ij} \Rr| .
\end{equation*}
Lemma~\ref{l.concentration} gives us that
\begin{equation*}  
\E \Ll[ \Ll| \|1A\|_{\ell^\infty \to \ell^\infty}  - \E [\|1A\|_{\ell^\infty \to \ell^\infty} ] \Rr| ^{2} \Rr] \le C n_0 n_1.
\end{equation*}
By Chebyshev's inequality, it therefore suffices to show that 
\begin{equation*}  
n_0^{-\frac 1 2} \E \Ll|\sum_{i = 1}^{n_0} A_{ij} \Rr| \xrightarrow[n_1 \to +\infty]{} \sqrt{\frac 2 \pi}.
\end{equation*}
This follows from the central limit theorem and a truncation argument. (If a moment of order higher than $2$ is assumed, this also follows directly from Proposition~\ref{p.gaussian}.)
\end{proof}

We now turn to the proof of Theorem~\ref{t.1P} per se. 

\begin{proof}[Proof of Theorem~\ref{t.1P}] The case $p = 1$ involving a single random matrix is covered by Proposition~\ref{p.1A}, so we assume $p \ge 2$ from now on.
We denote by $\E_p$ the expectation with respect to $A^{(p)}$ only, keeping all the other random variables fixed. We write $P'$ to denote the product of the first $p-1$ matrices, and use the shorthand notation $B := A^{(p)}$, so that $P = P'B$ and
\begin{equation*}  
\|1P\|_{\ell^\infty \to \ell^\infty} = \sum_{k = 1}^{n_p} \Ll| \sum_{i = 1}^{n_0} \sum_{j = 1}^{n_{p-1}} P'_{ij} B_{jk}  \Rr| .
\end{equation*}
We also impose the constraint \eqref{e.fix.magnitudes} on the integers $n_0, \ldots, n_p$, and write $n := n_0$ for short. Within this proof, the constant $C < +\infty$ is allowed to depend on the laws of the entries of the random matrices.
Using Lemma~\ref{l.concentration} with respect to $\E_p$,  we see that
\begin{equation*}  
\E_p \Ll[ \Ll|\|1P\|_{\ell^\infty \to \ell^\infty} - \E_p [\|1P\|_{\ell^\infty \to \ell^\infty}]\Rr|^2 \Rr] \le C n_p \sum_{j = 1}^{n_{p-1}}  \Ll( \sum_{i = 1}^{n_0} P'_{ij} \Rr) ^2.
\end{equation*}
Using Lemma~\ref{l.concentr.1A1.1}, we deduce that 
\begin{equation*}  
\E \Ll[ \Ll|\|1P\|_{\ell^\infty \to \ell^\infty} - \E_p [\|1P\|_{\ell^\infty \to \ell^\infty}]\Rr|^2 \Rr] \le C n^{p+1}.
\end{equation*}
By Chebyshev's inequality, we can thus study the asymptotics of $\E_p [\|1P\|_{\ell^\infty \to \ell^\infty}]$ in place of $\|1P\|_{\ell^\infty \to \ell^\infty}$, and we have
\begin{equation*}  
\E_p [\|1P\|_{\ell^\infty \to \ell^\infty}] = n_p \E_p \Ll| \sum_{i = 1}^{n_0} \sum_{j = 1}^{n_{p-1}} P'_{ij} B_{j1}  \Rr| .
\end{equation*}
By Proposition~\ref{p.gaussian}, we have that 
\begin{multline*}  
\Ll| \Ll( \sum_{j = 1}^{n_{p-1}} \Ll(\sum_{i = 1}^{n_0} P'_{ij}\Rr)^2 \Rr)^{-\frac 1 2} \E_p \Ll| \sum_{i,j} P'_{ij} B_{j1} \Rr| - \sqrt{\frac{2}{\pi}}  \Rr| \\
\le C \Ll( \sum_{j = 1}^{n_{p-1}} \Ll(\sum_{i = 1}^{n_0} P'_{ij}\Rr)^2 \Rr)^{-\al} \sum_{j = 1}^{n_{p-1}} \Ll| \sum_{i = 1}^{n_0} P'_{ij}\Rr|^{2\al}.
\end{multline*}
The estimates \eqref{e.tsum.1} and \eqref{e.tsum.2}, themselves derived from Lemmas~\ref{l.concentr.1A1.1} and \ref{l.concentr.1A1.2} respectively, allow us to conclude. 
\end{proof}

%
%
%
%
%
%
\section{Products ending with 1}
\label{s.P1}

This section is focused on the proof of Theorem~\ref{t.P1}. That is, we consider products of random matrices with centered entries to which we append a $\1$-type matrix at the end. In the previous base cases, the fact that the expression for the operator norm in \eqref{e.def.norm} involves a maximum was dealt with using concentration estimates, so that the asymptotic behavior of the maximum was no different than, say, the term indexed by $i = 1$. The case studied in this section is fundamentally different, as the maximum creates a genuine change of behavior witnessed by an additional logarithmic divergence. The workhorse of the proof of Theorem~\ref{t.P1} is the following proposition.
\begin{proposition}
\label{p.A1}
Let $(B_{ij})_{i,j \ge 1}$ be a family of i.i.d.\ centered random variables with unit variance and finite fourth moment. For all real numbers $(a_{j,n})_{1 \le j \le n}$ satisfying
\begin{equation}  
\label{e.A1.ass.a}
\sum_{j = 1}^n a_{j,n}^2 = n,  \qquad  \forall j\le n, \ |a_{j,n}| \le n^{\frac 1 {20}},
\end{equation}
and 
\begin{equation}
\label{e.A1.ass.b}
\lim_{n \to \infty} \frac{1}{(\log n)^2 n}  \sum_{j = 1}^n a_{j,n}^4  = 0,
\end{equation}
we have
\begin{equation}  
\label{e.A1}
(2n \log m)^{-\frac 1 2} \max_{1 \le i \le m} \Ll| \sum_{j = 1}^{n}  B_{ij} a_{j,n} \Rr| \xrightarrow[m \asymp n \to \infty]{\text{(prob.)}} 1.  
\end{equation}
\end{proposition}
\begin{proof}
We fix $C_0 < +\infty$ and only consider choices of $m,n$ that satisfy $C_0^{-1} m \le n \le C_0 m$.
We define
\begin{equation*}  
\td B_{ij} := B_{ij} \1_{\{|B_{ij}| \le n^{1/4}\}}, \qquad 
\hat B_{ij} := B_{ij} \1_{\{|B_{ij}| > n^{1/4}\}},
\end{equation*}
and 
\begin{equation*}  
\ov B_{ij} := \td B_{ij} - \E \Ll[ \td B_{ij} \Rr] .
\end{equation*}
Although this is kept implicit in the notation, we stress that these random variables depend on $n$. We decompose the rest of the proof into four steps.

\medskip

\noindent \emph{Step 1.} In this step, we show that in the expression on the left side of~\eqref{e.A1}, we may as well substitute $B_{ij}$ with $\ov B_{ij}$. 
Since $B_{ij}$ is centered, we have 
\begin{equation*}  
B_{ij} = \td B_{ij} + \hat B_{ij} = \td B_{ij} - \E[\td B_{ij}] + \hat B_{ij} - \E[\hat B_{ij}] = \bar B_{ij} + \hat B_{ij} - \E[\hat B_{ij}].
\end{equation*}
Moreover,
\begin{multline*}  
\E \Ll[\Ll| \sum_{j = 1}^{n} a_{j,n}(\hat B_{ij} - \E[\hat B_{ij}]) \Rr|^4 \Rr] 
\\
\le \sum_{j = 1}^{n} a_{j,n}^4 \E \Ll[ |\hat B_{ij} - \E[\hat B_{ij}]|^4 \Rr] + 6 \Ll( \sum_{j = 1}^{n} a_{j,n}^2\E[|\hat B_{ij} - \E[\hat B_{ij}]|^2] \Rr) ^2.
\end{multline*}
By Chebyshev's inequality, we have
\begin{equation*}  
\Ll|\E \Ll[ \hat B_{ij} \Rr] \Rr| \le \frac{\E[|B_{ij}|^4]}{n} \quad \text{ and } \quad \E\Ll[|\hat B_{ij}|^2\Rr] \le \frac{\E[|B_{ij}|^4]}{\sqrt n},
\end{equation*}
so
\begin{equation*}  
\E \Ll[\Ll| \sum_{j = 1}^{n} a_{j,n}(\hat B_{ij} - \E[\hat B_{ij}]) \Rr|^4 \Rr] \le C \sum_{j = 1}^n a_{j,n}^4.
\end{equation*}
It therefore follows that, for every $\ep > 0$, 
\begin{equation*}  
\P \Ll[ \Ll| \sum_{j = 1}^{n} a_{j,n} (\hat B_{ij} - \E[\hat B_{ij}]) \Rr| \ge \ep \sqrt{n \log m} \Rr] \le C (n \ep^2 \log m)^{-2} \sum_{j = 1}^n a_{j,n}^4,
\end{equation*}
and by a union bound, we obtain that 
\begin{equation*}  
\P \Ll[ \max_{1 \le i \le m} \Ll| \sum_{j = 1}^{n} a_{j,n}(\hat B_{ij} - \E[\hat B_{ij}]) \Rr| \ge \ep \sqrt{n \log m} \Rr] \le C n^{-1} \ep^{-4} (\log m)^{-2} \sum_{j = 1}^n a_{j,n}^4.
\end{equation*}
In order to prove the result, it therefore suffices to show that 
\begin{equation}
\label{e.A1.goal}
(2 n \log m)^{-\frac 1 2} \max_{1 \le i \le m} \Ll| \sum_{j = 1}^{n}  \bar B_{ij} a_{j,n}  \Rr| \xrightarrow[m \asymp n \to \infty]{\text{(prob.)}} 1.
\end{equation}

\medskip

\noindent \emph{Step 2.}
We recall that the law of $\bar B_{11}$ depends on $n$. For every $\lambda \in \R$, we define
\begin{equation*}  
\psi_n(\lambda) := \log \E \Ll[ e^{\lambda \bar B_{11}} \Rr] .
\end{equation*}
The goal of this step is to work out a Taylor expansion of $\psi_n$ near the origin. We denot by $(\kappa_{k,n})_{k \ge 1}$ the cumulants of $\bar B_{11}$, so that 
\begin{equation*}  
\psi_n(\lambda) = \sum_{k = 1}^{+\infty} \kappa_{k,n} \frac{\lambda^k}{k!}. 
\end{equation*}
Since $B_{11}$ has finite fourth moment, for every $k \in \{1,2,3,4\}$, we can define $\kappa_k$ the $k$-th cumulant of $B_{11}$, and we have $\lim_{n \to +\infty} \kappa_{k,n} = \kappa_k$. For any bounded random variable $X$, we write
\begin{equation*}  
\E_{\lambda} [X] := \frac{\E[X e^{\lambda \bar B_{11}}]}{\E[e^{\lambda \bar B_{11}}]}
\end{equation*}
Although the notation does not show it, this expectation depends on $n$. We have 
\begin{equation*}  
\psi_n'(\lambda ) = \E_{\lambda}[\bar B_{11}],
\end{equation*}
and we can compute the subsequent derivatives of $\psi_n$ by induction using the observation that, for every integer $k \ge 1$,  
\begin{equation*}  
\frac {\d}{\d \lambda} \E_{\lambda} [\bar B_{11}^k] = \E_{\lambda} [\bar B_{11}^{k+1}]  - \E_{\lambda} [\bar B_{11}^k]  \E_{\lambda} [\bar B_{11}].
\end{equation*}
The $k$-th derivative of $\psi_n$ is therefore a polynomial in $(\E_{\lambda} [\bar B_{11}^\ell])_{\ell \le k}$, and it is homogeneous in $\bar B_{11}$ of degree $k$. In particular, for each $k$, there exists a constant $C_k < +\infty$ such that $|\psi_n^{(k)}| \le C_k n^{\frac k 4}$. Observing also that $\kappa_{1,n} = 0$, we thus have that, for every $\lambda \in \R$, 
\begin{equation}
\label{e.A1.taylor.psi}
\Ll| \psi_n(\lambda) - \sum_{k = 2}^4 \kappa_{k,n} \frac{\lambda^k}{k!}\Rr| \le C_5 |\lambda|^5 n^{\frac 5 4}.
\end{equation}
In the sequel, we will want to make use of this expansion for $\lambda$ of the order of $\sqrt{(\log n)/n}$. Recall also that $\lim_{n \to \infty} \kappa_{2,n} =  \kappa_2 = 1$. In place of \eqref{e.A1.taylor.psi}, it will suffice for us to notice that in the limit $n \to \infty$, $\lambda \to 0$ with $|\lambda| = o(n^{-\frac 5 {12}})$, we have that 
\begin{equation}
\label{e.expand.psi}
\psi_n(\lambda) = \frac{\lambda^2}{2}(1+o(1)).
\end{equation}

\medskip

\noindent \emph{Step 3.} In this step, we prove the upper bound in our goal of showing \eqref{e.A1.goal}. More precisely, fixing $\ep > 0$, we show that 
\begin{equation}  
\label{e.A1.upper}
\P \Ll[ \Ll| \sum_{j = 1}^{n}a_{j,n} \bar B_{ij} \Rr| \ge (1+\ep) \sqrt{2n \log n} \Rr] = o \Ll( n^{-1} \Rr) .
\end{equation}
To show \eqref{e.A1.upper}, we use Chebyshev's inequality to write, for $\lambda \ge 0$, 
\begin{equation*}  
\P \Ll[ \sum_{j = 1}^{n}a_{j,n} \bar B_{ij}  \ge (1+\ep) \sqrt{2n \log n} \Rr] \le \exp \Ll(  - \lambda (1+\ep) \sqrt{2n \log n} +\sum_{j = 1}^n \psi_n(a_{j,n} \lambda) \Rr) .
\end{equation*}
We select $\lambda = (1+\ep)\sqrt{2(\log n)/n}$ and use the expansion in \eqref{e.expand.psi} and the assumptions in \eqref{e.A1.ass.a} to obtain that 
\begin{equation*}  
\P \Ll[ \sum_{j = 1}^{n} a_{j,n}\bar B_{ij}  \ge (1+\ep) \sqrt{2n \log n} \Rr] \le \exp \Ll( -(1+\ep)^2 (1+o(1))\log n \Rr) .
\end{equation*}
The same argument also applies to the estimatation of the lower deviations of the sum, and we thus obtain \eqref{e.A1.upper}.

\medskip 

\noindent \emph{Step 4.} In this step, we prove the lower bound corresponding to \eqref{e.A1.goal} and thereby complete the proof. We introduce the events $\mcl E_+$ and $\mcl E_-$ defined respectively by
\begin{equation*}  
\sum_{j = 1}^{n} a_{j,n} \bar B_{ij}  \ge  \sqrt{2n \log n} \quad \text{ and }\quad \sum_{j = 1}^{n} a_{j,n} \bar B_{ij}  \le (1-2\ep) \sqrt{2n \log n} ,
\end{equation*}
as well as $\mcl E := \mcl E_+^c \cap \mcl E_-^c$ (with $E^c$ denoting the complementary of $E$). 
For $\lambda_n := (1-\ep)\sqrt{2(\log n)/n}$, we write
\begin{equation}  
\label{e.A1.lower}
\P \Ll[ \mcl E \Rr] \ge \E \Ll[ \exp \Ll(\lambda_n \sum_{j = 1}^{n} a_{j,n} \bar B_{ij} \Rr) \1_{\mcl E} \Rr] \exp \Ll( -\lambda_n  \sqrt{2n \log n} \Rr).
\end{equation}
We aim to show that, up to a multiplicative error of $1+o(1)$, we can remove the indicator function $\1_{\mcl E}$ in \eqref{e.A1.lower}. For $\mu_n := \ep \sqrt{2(\log n)/n} \ge 0$, we have
\begin{align*}  
& \E \Ll[ \exp \Ll(\lambda_n \sum_{j = 1}^{n}a_{j,n}  \bar B_{ij} \Rr) \1_{\mcl E_+} \Rr] 
\\
& \qquad \le \E \Ll[ \exp \Ll((\lambda_n + \mu_n) \sum_{j = 1}^{n} a_{j,n} \bar B_{ij} \Rr)  \Rr] \exp \Ll( -\mu_n  \sqrt{2 n \log n} \Rr) 
\\
& \qquad \le \exp\Ll(- \mu_n  \sqrt{2 n \log n} + \sum_{j = 1}^n \psi_n(a_{j,n}(\lambda_n + \mu_n)) \Rr)
\\
& \qquad \le \exp \Ll( (1-2\ep)(1+o(1))\log n \Rr) .
\end{align*}
For $\mcl E_-$, a similar argument yields that
\begin{align*}  
& \E \Ll[ \exp \Ll(\lambda_n \sum_{j = 1}^{n} a_{j,n} \bar B_{ij} \Rr) \1_{\mcl E_-} \Rr] 
\\
& \qquad \le \E \Ll[ \exp \Ll((\lambda_n - \mu_n) \sum_{j = 1}^{n} a_{j,n} \bar B_{ij} \Rr)  \Rr] \exp \Ll( \mu_n (1-2\ep) \sqrt{2 n \log n} \Rr) 
\\
& \qquad \le \exp\Ll(\mu_n(1-2\ep)  \sqrt{2 n \log n} + \sum_{j = 1}^n \psi_n (a_{j,n}(\lambda_n - \mu_n))\Rr)
\\
& \qquad \le \exp \Ll( (1-2\ep)(1+o(1))\log n \Rr) .
\end{align*}
Since
\begin{equation*}  
\E \Ll[ \exp \Ll(\lambda_n \sum_{j = 1}^{n} a_{j,n} \bar B_{ij} \Rr) \Rr] = \exp \Ll(\sum_{j = 1}^n \psi_n(a_{j,n} \lambda) \Rr) = \exp \Ll( (1-\ep)^2(1+o(1))\log n \Rr) ,
\end{equation*}
we can combine these estimates with \eqref{e.A1.lower} to obtain that 
\begin{align*}  
\P[\mcl E] & \ge  \exp \Ll((1-\ep)^2 (1+o(1))\log n  - \lambda_n \sqrt{2 n \log n}\Rr) \\
& \ge \exp \Ll( -(1-\ep^2) (1+o(1))\log n  \Rr) .
\end{align*}
To sum up, we have obtained that $n \P[\mcl E]$ diverges to infinity like a small power of $n$. By independence, we also have that 
\begin{equation*}  
\P \Ll[ \max_{1 \le i \le m} \sum_{j = 1}^{n} a_{j,n} \bar B_{ij} > (1-2\ep) \sqrt{2n \log n} \Rr] \ge 1-(1-\P[\mcl E])^n,
\end{equation*}
and thus this probability tends to $1$ as $n$ tends to infinity, as desired.
\end{proof}

\begin{remark}  
\label{r.random.coefs}
To further prepare the ground for the proof of Theorem~\ref{t.P1}, we note that the following strengthening of Proposition~\ref{p.A1} is also valid. With the same assumptions on $B$ as in Proposition~\ref{p.A1}, suppose now that the coefficients $(a_{j,n})$ are random, independent of $B$, and satisfy
\begin{equation*}  
\sum_{j = 1}^n a_{j,n}^2 = n  \quad \text{a.s.}, 
\end{equation*}
\begin{equation*}  
\P\Ll[2n^{-\frac 1 {20}} \max_{j \le n} |a_{j,n}|  \le 1\Rr] \xrightarrow[n \to \infty]{} 1,
\end{equation*}
and
\begin{equation*}  
\frac{1}{(\log n)^2 n}  \sum_{j = 1}^n a_{j,n}^4  \xrightarrow[n \to \infty]{\text{(prob.)}} 0.
\end{equation*}
We claim that under these assumptions, the conclusion \eqref{e.A1} of Proposition~\ref{p.A1} still holds. One possible way to see this is to simply inspect the proof given above. Alternatively, one can appeal to Skorokhod's representation theorem and build a new probability space for the variables $(a_{j,n})$ under which the condition \eqref{e.A1.ass.a} holds almost surely for $n$ sufficiently large, and the condition \eqref{e.A1.ass.b} holds almost surely. With this construction, we can appeal to Proposition~\ref{p.A1} as stated to conclude for the stronger result. 
\end{remark}
\begin{proof}[Proof of Theorem~\ref{t.P1}]
We have
\begin{equation}  
\label{e.pf.P1}
m^{-1} \|P \1_{n_p,m} \|_{\ell^\infty \to \ell^\infty} = \max_{1 \le i \le n_0} \Ll| \sum_{j = 1}^{n_p} P_{ij} \Rr| .
\end{equation}
We argue by induction on $p$. 
In the case $p = 1$, a direct application of Proposition~\ref{p.A1} allows us to obtain the announced result. From now on, we assume that $p \ge 2$, let $P'$ denote the product of the $p-1$ last matrices, and write $\P_1$, $\E_1$ to denote respectively the probability and the expectation with respect to $A^{(1)}$ only, keeping the other variables fixed. We also use the shorthand notation $B := A^{(1)}$, so that $P = B P'$. We also impose the constraint \eqref{e.fix.magnitudes} on the integers $n_0, \ldots, n_p$. The quantity on the right side of \eqref{e.pf.P1} can be rewritten as
\begin{equation} 
\label{e.P1.goal}
\max_{1 \le i \le n_0} \Ll| \sum_{j = 1}^{n_1} \sum_{k = 1}^{n_p} B_{ij} P'_{jk} \Rr| .
\end{equation}
Applying Lemmas~\ref{l.concentr.1A1.2} and \ref{l.concentr.1A1.1} to the transpose of $P'$ yield, respectively, that 
\begin{equation*}  
(n_1 \, \cdots \, n_p)^{-1} \sum_{j = 1}^{n_1} \Ll| \sum_{k = 1}^{n_p} P'_{jk} \Rr| ^2 \xrightarrow[n_1 \asymp \cdots \asymp n_p \to \infty]{\text{(prob.)}}1
\end{equation*}
and
\begin{equation*}  
  \E \Ll[\sum_{j = 1}^{n_1}\Ll| \sum_{k = 1}^{n_p} P'_{jk} \Rr| ^4 \Rr] \le  C n_1 (n_2 \, \cdots \, n_p)^{2}.
\end{equation*}
By the induction hypothesis, the event that
\begin{equation*}  
\max_{1 \le j \le n_1} \Ll| \sum_{k = 1}^{n_p} P_{jk}'\Rr| \le 2 (n_2 \, \cdots \, n_p \log n_1)^\frac 1 2
\end{equation*}
has probability tending to $1$ as $n_1 \asymp \cdots \asymp n_p \to \infty$. In view also of Remark~\ref{r.random.coefs}, we are thus in a position to appeal to Proposition~\ref{p.A1} and conclude for the desired result. 
\end{proof}

\medskip

\noindent \textbf{Acknowledgements.} I would like to thank Charles Bordenave, Thomas Buc-d'Alché, Guillaume Dubach, Cédric Gerbelot and Rémi Gribonval for stimulating and helpful discussions. I acknowledge the support of the ANR-24-RRII-0002 grant operated by the Inria Quadrant program and part of the France 2030 plan.

\small
\bibliographystyle{plain}
\bibliography{matrix_norms}

\end{document}